\documentclass[a4paper,english,fontsize=10pt,parskip=half,abstracton]{scrartcl}
\usepackage{babel}
\usepackage[utf8]{inputenc}
\usepackage[T1]{fontenc}
\usepackage[a4paper,left=20mm,right=20mm,top=30mm,bottom=30mm]{geometry}
\usepackage{amsmath}
\usepackage{amsthm}
\usepackage{amssymb}
\usepackage{enumerate}
\usepackage[bookmarks=false,
            pdftitle={Blocks with defect group D2n x C2m},
            pdfauthor={Benjamin Sambale},
            pdfkeywords={blocks, defect groups, dihedral groups, fusion systems, Alperin's conjecture},
            pdfstartview={FitH}]{hyperref}
\newtheorem{Theorem}{Theorem}[section]
\newtheorem{Lemma}[Theorem]{Lemma}

\renewcommand{\phi}{\varphi}
\newcommand{\C}{\operatorname{C}}
\newcommand{\A}{\operatorname{A}}
\newcommand{\N}{\operatorname{N}}
\newcommand{\Z}{\operatorname{Z}}
\newcommand{\cohom}{\operatorname{H}}
\newcommand{\Aut}{\operatorname{Aut}}
\newcommand{\Out}{\operatorname{Out}}
\newcommand{\pcore}{\operatorname{O}}
\newcommand{\GL}{\operatorname{GL}}
\newcommand{\Irr}{\operatorname{Irr}}

\newcommand{\Gal}{\operatorname{Gal}}

\mathchardef\ordinarycolon\mathcode`\:  
 \mathcode`\:=\string"8000
 \begingroup \catcode`\:=\active
   \gdef:{\mathrel{\mathop\ordinarycolon}}
 \endgroup

\title{Blocks with defect group $D_{2^n}\times C_{2^m}$}
\author{
Benjamin Sambale\\
Mathematisches Institut\\
Friedrich-Schiller-Universität\\
07743 Jena\\
Germany\\
{\tt benjamin.sambale@uni-jena.de}
}
\date{\today}

\begin{document}
\frenchspacing
\maketitle
\begin{abstract}\noindent
We determine the numerical invariants of blocks with defect group $D_{2^n}\times C_{2^m}$, where $D_{2^n}$ denotes a dihedral group of order $2^n$ and $C_{2^m}$ denotes a cyclic group of order $2^m$. This generalizes Brauer's results \cite{Brauer} for $m=0$. As a consequence, we prove Brauer's $k(B)$-conjecture, Olsson's conjecture (and more generally Eaton's conjecture), Brauer's height zero conjecture, the Alperin-McKay conjecture, Alperin's weight conjecture and Robinson's ordinary weight conjecture for these blocks. Moreover, we show that the gluing problem has a unique solution in this case.
\end{abstract}


\section{Introduction}
Let $R$ be a discrete complete valuation ring with quotient field $K$ of characteristic $0$. Moreover, let $(\pi)$ be the maximal ideal of $R$ and $F:=R/(\pi)$. We assume that $F$ is algebraically closed of characteristic $2$.
We fix a finite group $G$, and assume that $K$ contains all $|G|$-th roots of unity. Let $B$ be a $2$-block of $RG$ with defect group $D$. We denote the number of irreducible ordinary characters of $B$ by $k(B)$. These characters split in $k_i(B)$ characters of height $i\in\mathbb{N}_0$. Here the \emph{height} of a character $\chi$ in $B$ is the largest integer $h(\chi)\ge 0$ such that $2^{h(\chi)}|G:D|_2\mathrel{\big|}\chi(1)$, where $|G:D|_2$ denotes the highest $2$-power dividing $|G:D|$.
Finally, let $l(B)$ be the number of irreducible Brauer characters of $B$.

If $D$ is a dihedral group, then all invariants of $B$ are known (see \cite{Brauer}). Thus, it seems natural to consider the case, where $D$ is a direct product of a dihedral group and a cyclic group. We write
\[D:=\langle x,y,z\mid x^{2^{n-1}}=y^2=z^{2^m}=[x,z]=[y,z]=1,\ yxy^{-1}=x^{-1}\rangle=\langle x,y\rangle\times\langle z\rangle\cong D_{2^n}\times C_{2^m},\]
where $n\ge 2$ and $m\ge0$. In the case $n=2$ and $m=0$ we get a four-group. Then the invariants of $B$ have been known for a long time. If $n=2$ and $m=1$, $D$ is elementary abelian of order $8$, and the block invariants are also known (see \cite{KKL}).
Finally, in the case $n=2\le m$ there exists a perfect isometry between $B$ and its Brauer correspondent (see \cite{UsamiZ4}). Thus, also in this case the block invariants are known, and the major conjectures are satisfied. Hence, we assume $n\ge 3$ for the rest of the paper. We allow $m=0$, since the results are completely consistent in this case.

In contrast to Brauer's work we use a more modern language and give shorter proofs. In addition we apply the theory of lower defect groups and the theory of centrally controlled blocks (see \cite{KuelshammerOkuyama}). The main reason that these blocks are accessible lies in the fact that certain inequalities for $k(B)$ and $k_i(B)$ are sharp.

\section{Subsections}
\begin{Lemma}\label{aut}
The automorphism group $\Aut(D)$ is a $2$-group.
\end{Lemma}
\begin{proof}
This is known for $m=0$. For $m\ge1$ the subgroups $\Phi(D)<\Phi(D)\Z(D)<\langle x,z\rangle<D$ are characteristic in $D$. By Theorem~5.3.2 in \cite{Gorenstein} every automorphism of $\Aut(D)$ of odd order acts trivially on $D/\Phi(D)$. The claim follows from Theorem~5.1.4 in \cite{Gorenstein}.
\end{proof}

It follows that the inertial index $e(B)$ of $B$ equals $1$. Now we investigate the fusion system $\mathcal{F}$ of the $B$-subpairs. For this we use the notation of \cite{Olssonsubpairs,Linckelmann}, and we assume that the reader is familiar with these articles. Let $b_D$ be a Brauer correspondent of $B$ in $RD\C_G(D)$. Then for every subgroup $Q\le D$ there is a unique block $b_Q$ of $RQ\C_G(Q)$ such that $(Q,b_Q)\le(D,b_D)$. We denote the inertial group of $b_Q$ in $\N_G(Q)$ by $\N_G(Q,b_Q)$.

\begin{Lemma}\label{essential}
Let $Q_1:=\langle x^{2^{n-2}},y,z\rangle\cong C_2^2\times C_{2^m}$ and $Q_2:=\langle x^{2^{n-2}},xy,z\rangle\cong C_2^2\times C_{2^m}$. 
Then $Q_1$ and $Q_2$ are the only candidates for proper $\mathcal{F}$-centric, $\mathcal{F}$-radical subgroups up to conjugation. In particular the fusion of subpairs is controlled by $\N_G(Q_1,b_{Q_1})\cup\N_G(Q_2,b_{Q_2})\cup D$.
Moreover, one of the following cases occurs:
\begin{enumerate}
\item[(aa)] $\N_G(Q_1,b_{Q_1})/\C_G(Q_1)\cong S_3$ and $\N_G(Q_2,b_{Q_2})/\C_G(Q_2)\cong S_3$.
\item[(ab)] $\N_G(Q_1,b_{Q_1})=\N_D(Q_1)\C_G(Q_1)$ and $\N_G(Q_2,b_{Q_2})/\C_G(Q_2)\cong S_3$.
\item[(ba)] $\N_G(Q_1,b_{Q_1})/\C_G(Q_1)\cong S_3$ and $\N_G(Q_2,b_{Q_2})=\N_D(Q_2)\C_G(Q_2)$.
\item[(bb)] $\N_G(Q_1,b_{Q_1})=\N_D(Q_1)\C_G(Q_1)$ and $\N_G(Q_2,b_{Q_2})=\N_D(Q_2)\C_G(Q_2)$.
\end{enumerate}
In case (bb) the block $B$ is nilpotent.
\end{Lemma}
\begin{proof}
Let $Q<D$ be $\mathcal{F}$-centric and $\mathcal{F}$-radical. Then $z\in\Z(D)\subseteq\C_D(Q)\subseteq Q$ and $Q=(Q\cap\langle x,y\rangle)\times\langle z\rangle$. Since $\Aut(Q)$ is not a $2$-group, $Q\cap\langle x,y\rangle$ and thus $Q$ must be abelian (see Lemma~\ref{aut}). Let us consider the case $Q=\langle x,z\rangle$. Then $m=n-1$ (this is not important here).
The group $D\subseteq\N_G(Q,b_Q)$ acts trivially on $\Omega(Q)\subseteq\Z(D)$, while a nontrivial automorphism of $\Aut(Q)$ of odd order acts nontrivially on $\Omega(Q)$ (see Theorem~5.2.4 in \cite{Gorenstein}). This contradicts $\pcore_2(\Aut_{\mathcal{F}}(Q))=1$. Hence, $Q$ is isomorphic to $C_2^2\times C_{2^m}$, and contains an element of the form $x^iy$. After conjugation with a suitable power of $x$ we may assume $Q\in\{Q_1,Q_2\}$. This shows the first claim. 
The second claim follows from Alperin's fusion theorem.

Let $S\le D$ be an arbitrary subgroup isomorphic to $C_2^2\times C_{2^m}$. If $z\notin S$, the group $\langle S,z\rangle=(\langle S,z\rangle\cap\langle x,y\rangle)\times\langle z\rangle$ is abelian and of order at least $2^{m+3}$. Hence, $\langle S,z\rangle\cap\langle x,y\rangle$ would be cyclic. This contradiction shows $z\in S$. Thus, $S$ is conjugate to $Q\in\{Q_1,Q_2\}$. Since $\lvert\N_D(Q)\rvert=2^{m+3}$, we derive that $Q$ is fully $\mathcal{F}$-normalized (see Definition~2.2 in \cite{Linckelmann}). In particular $\N_D(Q)\C_G(Q)/\C_G(Q)\cong\N_D(Q)/Q\cong C_2$ is a Sylow $2$-subgroup of $\Aut_{\mathcal{F}}(Q)=\N_G(Q,b_Q)/\C_G(Q)$ by Proposition~2.5 in \cite{Linckelmann}. In particular $\pcore_{2'}(\Aut_{\mathcal{F}}(Q))$ has index $2$ in $\Aut_{\mathcal{F}}(Q)$. Assume $\N_D(Q)\C_G(Q)<\N_G(Q,b_Q)$. Lemma~5.4 in \cite{Linckelmann} shows $\pcore_2(\Aut_{\mathcal{F}}(Q))=1$. 
If $m\ne 1$, we have $\lvert\Aut(Q)\rvert=2^k\cdot3$ for some $k\in\mathbb{N}$, since $\Phi(Q)<\Omega(Q)\Phi(Q)\le Q$ are characteristic subgroups. Then $\Aut_{\mathcal{F}}(Q)=\N_G(Q,b_Q)/\C_G(Q)\cong S_3$. Hence, we may assume $m=1$. Then $\Aut_{\mathcal{F}}(Q)\le\Aut(Q)\cong\GL(3,2)$. Since the normalizer of a Sylow $7$-subgroup of $\GL(3,2)$ has order $21$, it follows that $\lvert\pcore_{2'}(\Aut_{\mathcal{F}}(Q))\rvert\ne 7$. Since this normalizer is selfnormalizing in $\GL(3,2)$, we also have $\lvert\pcore_{2'}(\Aut_{\mathcal{F}}(Q))\rvert\ne 21$. This shows $\lvert\pcore_{2'}(\Aut_{\mathcal{F}}(Q))\rvert=3$ and $\Aut_{\mathcal{F}}(Q)=\N_G(Q,b_Q)/\C_G(Q)\cong S_3$, because $\lvert\GL(3,2)\rvert=2^3\cdot3\cdot7$.

The last claim follows from Alperin's fusion theorem and $e(B)=1$.
\end{proof}

The naming of these cases is adopted from \cite{Brauer}. 
Since the cases (ab) and (ba) are symmetric, we ignore case (ba) for the rest of the paper. It is easy to see that $Q_1$ and $Q_2$ are not conjugate in $D$. Hence, by Alperin's fusion theorem the subpairs $(Q_1,b_{Q_1})$ and $(Q_2,b_{Q_2})$ are not conjugate in $G$. It is also easy to see that $Q_1$ and $Q_2$ are always $\mathcal{F}$-centric.

\begin{Lemma}\label{fixedpt}
Let $Q\in\{Q_1,Q_2\}$ such that $\N_G(Q,b_Q)/\C_G(Q)\cong S_3$. Then 
\[\C_Q(\N_G(Q,b_Q))\in\{\langle z\rangle,\ \langle x^{2^{n-2}}z\rangle\}.\]
In particular $z^{2j}\in\C_Q(\N_G(Q,b_Q))$ and $x^{2^{n-2}}z^{2j}\notin\C_Q(\N_G(Q,b_Q))$ for $j\in\mathbb{Z}$.
\end{Lemma}
\begin{proof}
We consider only the case $Q=Q_1$ (the other case is similar). It is easy to see that the elements in $Q\setminus\Z(D)$ are not fixed under $\N_D(Q)\subseteq\N_D(Q,b_Q)$. 
Since $D$ acts trivially on $\Z(D)$, it suffices to determine the fixed points of an automorphism $\alpha\in\Aut_{\mathcal{F}}(Q)$ of order $3$ in $\Z(D)$. By Lemma~3.2 in \cite{Sambalemna} $\C_Q(\alpha)=\langle a\rangle$ has order $2^m$. First we show that $a\in\Z(D)$. Suppose the contrary.
Let $\beta\in\Aut_{\mathcal{F}}(Q)$ be the automorphism induced by $x^{2^{n-3}}\in\N_D(Q)\subseteq\N_G(Q,b_Q)$. Then we have $\beta(a)\ne a$.
Since $\beta\alpha\beta^{-1}=\alpha^{-1}$, we have $\alpha(\beta(a))=\beta(\alpha^{-1}(a))=\beta(a)$. Thus, $\beta(a)\in\C_Q(\alpha)=\langle a\rangle$. 
This gives the contradiction $\beta(a)a^{-1}\in D'\cap\langle a\rangle=\langle x^2\rangle\cap\langle a\rangle=1$.
Now in case $m\ne 1$ the claim is clear.
Thus, assume $m=1$ and $a=x^{2^{n-2}}$. Then $\beta$ acts trivially on $Q/\langle a\rangle$ and $\alpha$ acts nontrivially on $Q/\langle a\rangle$. This contradicts $\beta\alpha\beta^{-1}\alpha=1$.
\end{proof}

It is not possible to decide whether $\C_Q(\N_G(Q,b_Q))$ is $\langle z\rangle$ or $\langle x^{2^{n-2}}z\rangle$ in Lemma~\ref{fixedpt}, since we can replace $z$ by $x^{2^{n-2}}z$.
For a subgroup $Q\le D$ and an element $u\in\Z(Q)$ we write $b_u:=b_{\langle u\rangle}=b_Q^{\C_G(u)}$, where $b_Q^{\C_G(u)}$ denotes the Brauer correspondent of $b_Q$ in $R\C_G(u)$.

\begin{Lemma}\hfill\label{subrep}
\begin{enumerate}[(i)]
\item In case (aa) the subsections $(x^iz^j,b_{x^iz^j})$ ($i=0,1,\ldots,2^{n-2}$, $j=0,1,\ldots,2^m-1$) form a set of representatives for the conjugacy classes of $B$-subsections.\label{aarep}
\item In case (ab) the subsections $(x^iz^j,b_{x^iz^j})$ and $(yz^j,b_{yz^j})$ ($i=0,1,\ldots,2^{n-2}$, $j=0,1,\ldots,2^m-1$) form a set of representatives for the conjugacy classes of $B$-subsections.\label{abrep}
\end{enumerate}
\end{Lemma}
\begin{proof}
We investigate the set $\A_0(D,b_D)$ (see \cite{Olssonsubpairs}) and apply (6C) in \cite{Brauerstruc}. Since $D\in\A_0(D,b_D)$ and $e(B)=1$ there are $2^{m+1}$ major subsections $(z^j,b_{z^j})$ and  $(x^{2^{n-2}}z^j,b_{x^{2^{n-2}}z^j})$ ($j=0,1,\ldots,2^m-1$) which are pairwise nonconjugate. Now let $Q\in\A_0(D,b_D)$. As in the proof of Lemma~\ref{essential}, we have $Q=(Q\cap\langle x,y\rangle)\times\langle z\rangle$ (see Lemma~(3.1) in \cite{Olssonsubpairs}). If $Q\cap\langle x,y\rangle$ is a nonabelian dihedral group, then $\Z(Q)=\Z(D)$, and there are no subsections corresponding to $(Q,b_Q)$. On the other hand we have $Q:=\langle x,z\rangle\in\A_0(D,b_D)$ by Lemma~1.7 in \cite{Olsson}. 
Suppose that $\Aut_{\mathcal{F}}(Q)$ is not a $2$-group.
Then $m=n-1$ and $D\C_G(Q)/\C_G(Q)$ is a Sylow $2$-subgroup of $\Aut_{\mathcal{F}}(Q)$. Since $\Aut(D)$ is a $2$-group, Lemma~5.4 in \cite{Linckelmann} shows $\pcore_2(\Aut_{\mathcal{F}}(Q))=1$. 
However, this contradicts Lemma~\ref{essential}, since $Q$ is $\mathcal{F}$-centric.
This shows $\N_G(Q,b_Q)=D\C_G(Q)$. For a subsection $(u,b)$ with $u\in Q$ we must check whether $\lvert\N_G(Q,b_Q)\cap\C_G(u):Q\C_G(Q)\rvert$ is odd. It is easy to see that this holds if and only if $u\notin\Z(D)$. The action of $D$ on $Q\setminus\Z(D)$ gives the following subsections:
$(x^iz^j,b_{x^iz^j})$ ($i=1,\ldots,2^{n-2}-1$, $j=0,1,\ldots,2^m-1$).

Now suppose $Q=Q_2$ and $u\in Q\setminus\Z(D)$. Let $\alpha\in\Aut_{\mathcal{F}}(Q)$ be an automorphism of order $3$. As in the proof of Lemma~\ref{fixedpt} we have $\C_Q(\alpha)\subseteq\Z(D)$. Thus, $u\alpha(u)\alpha^{-1}(u)\in\C_Q(\alpha)\subseteq\Z(D)$. It follows that $\alpha(u)\in\Z(D)$ or $\alpha^{-1}(u)\in\Z(D)$, since $\Z(D)$ has index $2$ in $Q$. Let $\beta\in\Aut_{\mathcal{F}}(Q)$ be the automorphism induced by $x^{2^{n-3}}\in\N_D(Q)\subseteq\N_G(Q,b_Q)$. Then one of the $2$-elements $\alpha\beta\alpha^{-1}$ or $\alpha^{-1}\beta\alpha$ fixes $u$. This shows $2\mathrel{\big|}\lvert\N_G(Q,b_Q)\cap\C_G(u):\C_G(Q)\rvert$ for every $u\in Q$. Hence, there are no subsections corresponding to $(Q_2,b_{Q_2})$. In case (aa) the same holds for $(Q_1,b_{Q_1})$. This proves part~\eqref{aarep}. Let us consider $Q=Q_1$ in case (ab). By way of contradiction, suppose $Q\notin\A_0(D,b_D)$. Then we get the same set of representatives for the conjugacy classes of subsections as in case (aa). In particular the subpair $(\langle y\rangle,b_y)$ is conjugate to a subpair $(\langle u\rangle,b_u)$ with $u\in\Z(D)$. However, this contradicts Alperin's fusion theorem. Hence, $Q\in\A_0(D,b_D)$.
Then we have $\lvert\N_G(Q,b_Q)\cap\C_G(u):Q\C_G(Q)\rvert=\lvert\N_D(Q)\C_G(Q)\cap\C_G(u):\C_G(Q)\rvert=\lvert\C_G(Q)(\N_D(Q)\cap\C_G(u)):\C_G(Q)\rvert=\lvert\N_D(Q)\cap\C_G(u):Q\rvert$ for $u\in Q$. Thus, we have to take the subsections $(u,b)$ with $u\in Q\setminus\Z(D)$ up to $\N_D(Q)$-conjugation. This shows part~\eqref{abrep}. 
\end{proof}

\section{The numbers \texorpdfstring{$k(B)$}{k(B)}, \texorpdfstring{$k_i(B)$}{ki(B)} and \texorpdfstring{$l(B)$}{l(B)}}

Now we study the generalized decomposition numbers of $B$. If $l(b_u)=1$, then we denote the unique irreducible modular character of $b_u$ by $\phi_u$. In this case the generalized decomposition numbers $d^u_{\chi\phi_u}$ for $\chi\in\Irr(B)$ form a column $d(u)$. Let $2^k$ be the order of $u$, and let $\zeta:=\zeta_{2^k}$ be a primitive $2^k$-th root of unity. Then the entries of $d(u)$ lie in the ring of integers $\mathbb{Z}[\zeta]$. Hence, there exist integers $a_i^u:=(a_i^{u}(\chi))_{\chi\in\Irr(B)}\in\mathbb{Z}^{k(B)}$ such that
\[d_{\chi\phi_u}^u=\sum_{i=0}^{2^{k-1}-1}{a_i^{u}(\chi)\zeta^i}.\]
We extend this by
\[a_{i+2^{k-1}}^u:=-a_i^u\]
for all $i\in\mathbb{Z}$.

Let $|G|=2^ar$ where $2\nmid r$. We may assume $\mathbb{Q}(\zeta_{|G|})\subseteq K$. Then $\mathbb{Q}(\zeta_{|G|})\mid\mathbb{Q}(\zeta_r)$ is a Galois extension, and we denote the corresponding Galois group by
\[\mathcal{G}:=\Gal\bigl(\mathbb{Q}(\zeta_{|G|})\mid\mathbb{Q}(\zeta_r)\bigr).\] 
Restriction gives an isomorphism
\[\mathcal{G}\cong\Gal\bigl(\mathbb{Q}(\zeta_{2^a})\mid\mathbb{Q}\bigr).\]
In particular $|\mathcal{G}|=2^{a-1}$.
For every $\gamma\in\mathcal{G}$ there is a number $\widetilde{\gamma}\in\mathbb{N}$ such that $\gcd(\widetilde{\gamma},|G|)=1$, $\widetilde{\gamma}\equiv 1\pmod{r}$, and $\gamma(\zeta_{|G|})=\zeta_{|G|}^{\widetilde{\gamma}}$ hold. Then $\mathcal{G}$ acts on the set of subsections by
\[^\gamma(u,b):=(u^{\widetilde{\gamma}},b).\]
For every $\gamma\in\mathcal{G}$ we get
\begin{equation}\label{dgamma}
d(u^{\widetilde{\gamma}})=\sum_{s\in\mathcal{S}}{a_s^{u}\zeta_{2^k}^{s\widetilde{\gamma}}}
\end{equation}
for every system $\mathcal{S}$ of representatives of the cosets of $2^{k-1}\mathbb{Z}$ in $\mathbb{Z}$.
It follows that
\begin{equation}\label{aiuspur}
a_s^u=2^{1-a}\sum_{\gamma\in\mathcal{G}}{d\bigl(u^{\widetilde{\gamma}}\bigr)\zeta_{2^k}^{-\widetilde{\gamma}s}}
\end{equation}
for $s\in\mathcal{S}$. 

Next, we introduce a general result which does not depend on $D$.

\begin{Lemma}\label{heightzeroodd}
Let $(u,b_u)$ be a $B$-subsection with $|\langle u\rangle|=2^k$ and $l(b_u)=1$.
\begin{enumerate}[(i)]
\item If $\chi\in\Irr(B)$ has height $0$, then the sum 
\begin{equation}\label{sum}
\sum_{i=0}^{2^{k-1}-1}{a_i^u(\chi)}
\end{equation}
is odd.\label{ht1}
\item If $(u,b_u)$ is major and $k\le 1$, then $2^{h(\chi)}\mid d^u_{\chi\phi_u}=a_0^u(\chi)$ and $2^{h(\chi)+1}\nmid d^u_{\chi\phi_u}$ for all $\chi\in\Irr(B)$.\label{ht2}
\end{enumerate}
\end{Lemma}
\begin{proof}
Let $Q\le D$ be a defect group of $b_u$. Since $l(b_u)=1$, we have $|Q|m_{\chi\chi}^{(u,b_u)}=d^u_{\chi\phi_u}\overline{d^u_{\chi\phi_u}}$ for the contribution $m_{\chi\chi}^{(u,b_u)}$ (see Eq.~(5.2) in \cite{BrauerBlSec2}). Assume that $\chi$ has height $0$.
By Corollary~2 in \cite{BroueSanta} it follows that \[|Q|m_{\chi\chi}^{(u,b_u)}=|Q|\bigl(\chi^{(u,b_u)},\chi\bigr)_G\not\equiv 0\pmod{(\pi)}\] 
and $d^u_{\chi\phi_u}\not\equiv0\pmod{(\pi)}$. Since $\zeta_{2^k}\equiv 1\pmod{(\pi)}$, the sum~\eqref{sum} is odd.

Now assume that $(u,b_u)$ is major and $k\le 1$. Then $d^u_{\chi\phi_u}=a_0^u(\chi)\in\mathbb{Z}$ for all $\chi\in\Irr(B)$. If $\psi\in\Irr(B)$ has height $0$ ($\psi$ always exists), part~\eqref{ht1} shows that $d^u_{\psi\phi_u}$ is odd. By (5H) in \cite{BrauerBlSec2} we have $2^{h(\chi)}\mid|D|m_{\chi\psi}^{(u,b_u)}=d^u_{\chi\phi_u}d^u_{\psi\phi_u}$ and $2^{h(\chi)+1}\nmid|D|m_{\chi\psi}^{(u,b_u)}$. This proves part~\eqref{ht2}.
\end{proof}

\begin{Lemma}\label{olsson}
Olsson's conjecture $k_0(B)\le 2^{m+2}=|D:D'|$ is satisfied in all cases.
\end{Lemma}
\begin{proof}
Let $\gamma\in\mathcal{G}$ such that the restriction of $\gamma$ to $\mathbb{Q}(\zeta_{2^a})$ is the complex conjugation. Then $x^{\widetilde{\gamma}}=x^{-1}$. The block $b_x$ has defect group $\langle x,z\rangle$ (see the proof of (6F) in \cite{Brauerstruc}). Since we have shown that $\Aut_{\mathcal{F}}(\langle x,z\rangle)$ is a $2$-group, $b_x$ is nilpotent. In particular $l(b_x)=1$.
Since the subsections $(x,b_x)$ and $(x^{-1},b_{x^{-1}})=(x^{-1},b_x)={^{\gamma}(x,b_x)}$ are conjugate by $y$, we have $d(x)=d(x^{\widetilde{\gamma}})$ and 
\begin{equation}\label{aiequal}
a_j^x(\chi)=a_{-j}^x(\chi)=-a_{2^{n-2}-j}^x(\chi)                                                                    \end{equation} 
for all $\chi\in\Irr(B)$ by Eq.~\eqref{dgamma}. In particular $a_{2^{n-3}}^x(\chi)=0$ (cf. (4.16) in \cite{Brauer}). By the orthogonality relations we have $(d(x),d(x))=|\langle x,z\rangle|=2^{n-1+m}$. On the other hand the subsections $(x,b_x)$ and $(x^i,b_{x^i})=(x^i,b_x)$ are not conjugate for odd $i\in\{3,5,\ldots,2^{n-2}-1\}$. Eq.~\eqref{aiuspur} implies 
\[(a_0^x,a_0^x)=2^{2(1-a)}\sum_{\gamma,\delta\in\mathcal{G}}{\bigl(d(x^{\widetilde{\gamma}}),d(x^{\widetilde{\delta}})\bigr)}=2^{2(1-a)}2^{2a-n+1}(d(x),d(x))=2^{m+2}\]
(cf. Proposition~(4C) in \cite{Brauer}). Combining Eq.~\eqref{aiequal} with Lemma~\ref{heightzeroodd}\eqref{ht1} we see that $a_0^x(\chi)\ne 0$ is odd for characters $\chi\in\Irr(B)$ of height $0$. This proves the lemma.
\end{proof}

We remark that Olsson's conjecture in case (bb) also follows from Lemma~\ref{essential}. Moreover, in case (ab) Olsson's conjecture follows easily from Theorem~3.1 in \cite{Robinson}.

\begin{Theorem}\label{main}
In all cases we have 
\begin{align*}
k(B)&=2^m(2^{n-2}+3),&k_0(B)&=2^{m+2},&k_1(B)&=2^m(2^{n-2}-1).
\end{align*} 
Moreover,
\[l(B)=\begin{cases}1&\text{in case (bb)}\\2&\text{in case (ab)}\\3&\text{in case (aa)}\end{cases}.\]
In particular Brauer's $k(B)$-conjecture, Brauer's height zero conjecture and the Alperin-McKay conjecture hold.
\end{Theorem}
\begin{proof}
Assume first that case (bb) occurs. Then $B$ is nilpotent and $k_i(B)$ is just the number $k_i(D)$ of irreducible characters of $D$ of degree $2^i$ ($i\ge0$) and $l(B)=1$. Since $C_{2^m}$ is abelian, we get $k_i(B)=2^mk_i(D_{2^n})$. The claim follows in this case. Thus, we assume that case (aa) or case (ab) occurs.
We determine the numbers $l(b)$ for the subsections in Lemma~\ref{subrep} and apply (6D) in \cite{Brauerstruc}. Let us begin with the nonmajor subsections. Since $\Aut_{\mathcal{F}}(\langle x,z\rangle)$ is a $2$-group, the block $b_{\langle x,z\rangle}$ with defect group $\langle x,z\rangle$ is nilpotent. Hence, we have $l(b_{x^iz^j})=1$ for all $i=1,\ldots,2^{n-2}-1$ and $j=0,1,\ldots,2^m-1$. The blocks $b_{yz^j}$ ($j=0,1,\ldots,2^m-1$) have $Q_1$ as defect group. Since $\N_G(Q_1,b_{Q_1})=\N_D(Q_1)\C_G(Q_1)$, they are also nilpotent, and it follows that $l(b_{yz^j})=1$. 

We divide the (nontrivial) major subsections into three sets:
\begin{align*}
U&:=\{x^{2^{n-2}}z^{2j}:j=0,1,\ldots,2^{m-1}-1\},\\
V&:=\{z^j:j=1,\ldots,2^m-1\},\\
W&:=\{x^{2^{n-2}}z^{2j+1}:j=0,1,\ldots,2^{m-1}-1\}.
\end{align*}
By Lemma~\ref{fixedpt} case (bb) occurs for $b_u$, and we get $l(b_u)=1$ for $u\in U$.
The blocks $b_v$ with $v\in V$ dominate unique blocks $\overline{b_v}$ of $R\C_G(v)/\langle v\rangle$ with defect group $D/\langle v\rangle\cong D_{2^n}\times C_{2^m/|\langle v\rangle|}$ such that $l(b_v)=l(\overline{b_v})$ (see Theorem~5.8.11 in \cite{Nagao} for example). The same argument for $w\in W$ gives blocks $\overline{b_w}$ with defect group $D/\langle w\rangle\cong D_{2^n}$. This allows us to apply induction on $m$ (for the blocks $b_v$ and $b_w$).
The beginning of this induction ($m=0$) is satisfied by Brauer's result (see \cite{Brauer}). Thus, we may assume $m\ge 1$.
By Theorem~1.5 in \cite{Olsson} the cases for $b_v$ (resp. $b_w$) and $\overline{b_v}$ (resp. $\overline{b_w}$) coincide.

Suppose that case (ab) occurs. By Lemma~\ref{fixedpt} case (ab) occurs for exactly $2^m-1$ blocks in $\{b_v:v\in V\}\cup\{b_w:w\in W\}$ and case (bb) occurs for the other $2^{m-1}$ blocks.
Induction gives
\[\sum_{v\in V}{l(b_v)}+\sum_{w\in W}{l(b_w)}=\sum_{v\in V}{l(\overline{b_v})}+\sum_{w\in W}{l(\overline{b_w})}=2(2^m-1)+2^{m-1}.\]
Taking all subsections together, we derive
\[k(B)-l(B)=2^m(2^{n-2}+3)-2.\]
In particular $k(B)\ge 2^m(2^{n-2}+3)-1$. Let $u:=x^{2^{n-2}}\in\Z(D)$. Lemma~\ref{heightzeroodd}\eqref{ht2} implies $2^{h(\chi)}\mid d^u_{\chi\phi_u}$ and $2^{h(\chi)+1}\nmid d^u_{\chi\phi_u}$ for $\chi\in\Irr(B)$. In particular $d^u_{\chi\phi_u}\ne 0$.
Lemma~\ref{olsson} gives
\begin{equation}\label{k*ineq}
2^{n+m}-4\le k_0(B)+4(k(B)-k_0(B))\le\sum_{\chi\in\Irr(B)}{\bigl(d^u_{\chi\phi_u}\bigr)^2}=(d(u),d(u))=|D|=2^{n+m}. 
\end{equation}
Hence, we have 
\[d^u_{\chi\phi_u}=\begin{cases}\pm1&\text{if }h(\chi)=0\\\pm2&\text{otherwise}\end{cases},\] 
and the claim follows in case (ab).

Now suppose that case (aa) occurs. Then by the same argument as in case (ab) we have
\[\sum_{v\in V}{l(b_v)}+\sum_{w\in W}{l(b_w)}=\sum_{v\in V}{l(\overline{b_v})}+\sum_{w\in W}{l(\overline{b_w})}=3(2^m-1)+2^{m-1}.\]
Observe that this sum does not depend on which case actually occurs for $b_z$ (for example). In fact all three cases for $b_z$ are possible.
Taking all subsections together, we derive
\[k(B)-l(B)=2^m(2^{n-2}+3)-3.\]
Here it is not clear a priori whether $l(B)>1$. Brauer delayed the discussion of the possibility $l(B)=1$ until section~7 of \cite{Brauer}. Here we argue differently via lower defect groups and centrally controlled blocks. 
First we consider the case $m\ge2$. By Lemma~\ref{fixedpt} we have $\langle D,\N_G(Q_1,b_{Q_1}),\N_G(Q_2,b_{Q_2})\rangle\subseteq\C_G(z^2)$, i.\,e. $B$ is centrally controlled (see \cite{KuelshammerOkuyama}). By Theorem~1.1 in \cite{KuelshammerOkuyama} we get $l(B)\ge l(b_{z^2})=3$. Hence, the claim follows with Ineq.~\eqref{k*ineq}.

Now consider the case $m=1$. By Lemma~\ref{fixedpt} there is a (unique) nontrivial fixed point $u\in\Z(D)$ of $\N_G(Q_1,b_{Q_1})$. Then $l(b_u)>1$.
By Proposition~(4G) in \cite{Brauer} the Cartan matrix of $b_u$ has $2$ as an elementary divisor. With the notation of \cite{OlssonLDG} we have $m_{b_u}^{(1)}(Q)\ge1$ for some $Q\le\C_G(u)=\N_G(\langle u\rangle)$ with $|Q|=2$ (see the remark on page 285 in \cite{OlssonLDG}). In particular $Q$ is a lower defect group of $b_u$ (see Theorem~(5.4) in \cite{OlssonLDG}). Since $\langle u\rangle\le\Z(\C_G(u))$, Corollary~(3.7) in \cite{OlssonLDG} implies $Q=\langle u\rangle$. By Theorem~(7.2) in \cite{OlssonLDG} we have $m_B^{(1)}(\langle u\rangle)\ge 1$. In particular $2$ occurs as elementary divisor of the Cartan matrix of $B$. This shows $l(B)\ge 2$. Now the claim follows again with Ineq.~\eqref{k*ineq}.
\end{proof}

We add some remarks. For trivial reasons also Eaton's conjecture is satisfied which provides a generalization of Brauer's $k(B)$-conjecture and Olsson's conjecture (see \cite{Eaton}). Brauer's $k(B)$-conjecture already follows from Theorem~2 in \cite{SambalekB}. The principal blocks of $D$, $S_4\times C_{2^m}$ and $\GL(3,2)\times C_{2^m}$ give examples for the cases (bb), (ab) and (aa) respectively (at least for $n=3$). Moreover, the principal block of $S_6$ shows that also $\C_{Q_1}(\N_G(Q_1,b_{Q_1}))\ne\C_{Q_2}(\N_G(Q_2,b_{Q_2}))$ is possible in case (aa). This gives an example, where $B$ is not centrally controlled (and $m=1$). However, $B$ cannot be a block of maximal defect of a simple group for $m\ge 1$ by the main theorem in \cite{Haradasylow}.

\section{Alperin's weight conjecture}

Alperin's weight conjecture asserts that $l(B)$ is the number of conjugacy classes of weights for $B$. Here a weight is a pair $(Q,\beta)$, where $Q$ is a $2$-subgroup of $G$ and $\beta$ is a block of $R[\N_G(Q)/Q]$ with defect $0$. Moreover, $\beta$ is dominated by a Brauer correspondent $b$ of $B$ in $R\N_G(Q)$. 

\begin{Theorem}
Alperin's weight conjecture holds for $B$.
\end{Theorem}
\begin{proof}
We use Proposition~5.4 in \cite{Kessar}. For this, let $Q\le D$ be $\mathcal{F}$-centric and $\mathcal{F}$-radical. By Lemma~\ref{essential} we have $\Out_{\mathcal{F}}(Q)\cong S_3$ or $\Out_{\mathcal{F}}(Q)=1$ (if $Q=D$). In particular $\Out_{\mathcal{F}}(Q)$ has trivial Schur multiplier. Moreover, $F\Out_{\mathcal{F}}(Q)$ has precisely one block of defect $0$. Now the claim follows from Theorem~\ref{main} and Proposition~5.4 in \cite{Kessar}.
\end{proof}

\section{Ordinary weight conjecture}
In this section we prove Robinson's ordinary weight conjecture (OWC) for $B$ (see \cite{OWC}). If OWC holds for all groups and all blocks, then also Alperin's weight conjecture holds. However, for our particular block $B$ this implication is not known. In the same sense OWC is equivalent to Dade's projective conjecture (see \cite{Eaton}). Uno has proved Dade's invariant conjecture in the case $m=0$ (see \cite{UnoDade}). 
For $\chi\in\Irr(B)$ let $d(\chi):=n+m-h(\chi)$ be the \emph{defect} of $\chi$. We set $k^i(B)=|\{\chi\in\Irr(B):d(\chi)=i\}|$ for $i\in\mathbb{N}$. 

\begin{Theorem}
The ordinary weight conjecture holds for $B$.
\end{Theorem}
\begin{proof}
We prove the version in Conjecture~6.5 in \cite{Kessar}. For this, let $Q\le D$ be $\mathcal{F}$-centric and $\mathcal{F}$-radical. In the case $Q=D$ we have $\Out_{\mathcal{F}}(D)=1$ and $\mathcal{N}_D$ consists only of the trivial chain (with the notations of \cite{Kessar}). Then it follows easily that $\textbf{w}(D,d)=k^d(D)=k^d(B)$ for all $d\in\mathbb{N}$. Now let $Q\in\{Q_1,Q_2\}$ such that $\Out_{\mathcal{F}}(Q)=\Aut_{\mathcal{F}}(Q)\cong S_3$. It suffices to show that $\textbf{w}(Q,d)=0$ for all $d\in\mathbb{N}$. Since $Q$ is abelian, we have $\textbf{w}(Q,d)=0$ unless $d=m+2$. Thus, let $d=m+2$. Up to conjugation $\mathcal{N}_Q$ consists of the trivial chain $\sigma:1$ and the chain $\tau:1<C$, where $C\le\Out_{\mathcal{F}}(Q)$ has order $2$. 

We consider the chain $\sigma$ first. Here $I(\sigma)=\Out_{\mathcal{F}}(Q)\cong S_3$ acts faithfully on $\Omega(Q)\cong C_2^3$ and thus fixes a four-group. Hence, the characters in $\Irr(Q)$ split in $2^m$ orbits of length $3$ and $2^m$ orbits of length $1$ under $I(\sigma)$ (see also Lemma~\ref{fixedpt}). For a character $\chi\in\Irr(D)$ lying in an orbit of length $3$ we have $I(\sigma,\chi)\cong C_2$ and thus $w(Q,\sigma,\chi)=0$. For the $2^m$ stable characters $\chi\in\Irr(D)$ we get $w(Q,\sigma,\chi)=1$, since $I(\sigma,\chi)=\Out_{\mathcal{F}}(Q)$ has precisely one block of defect $0$.

Now consider the chain $\tau$. Here $I(\tau)=C$ and the characters in $\Irr(Q)$ split in $2^m$ orbits of length $2$ and $2^{m+1}$ orbits of length $1$ under $I(\tau)$. For a character $\chi\in\Irr(D)$ in an orbit of length $2$ we have $I(\tau,\chi)=1$ and thus $w(Q,\tau,\chi)=1$. For the $2^{m+1}$ stable characters $\chi\in\Irr(D)$ we get $I(\tau,\chi)=I(\tau)=C$ and $w(Q,\tau,\chi)=0$.

Taking both chains together, we derive
\[\textbf{w}(Q,d)=(-1)^{|\sigma|+1}2^m+(-1)^{|\tau|+1}2^m=2^m-2^m=0.\]
This proves OWC.
\end{proof}

\section{The gluing problem}
Finally we show that the gluing problem (see Conjecture~4.2 in \cite{gluingprob}) for the block $B$ has a unique solution. This was done for $m=0$ in \cite{Parkgluing}. We will not recall the very technical statement of the gluing problem. Instead we refer to \cite{Parkgluing} for most of the  notations. Observe that the field $F$ is denoted by $k$ in \cite{Parkgluing}.

\begin{Theorem}
The gluing problem for $B$ has a unique solution.
\end{Theorem}
\begin{proof}
We will show that $\cohom^i(\Aut_{\mathcal{F}}(\sigma),F^\times)=0$ for $i=1,2$ and every chain $\sigma$ of $\mathcal{F}$-centric subgroups of $D$. Then it follows that $\mathcal{A}_{\mathcal{F}}^i=0$ and $\cohom^0([S(\mathcal{F}^c)],\mathcal{A}^2_{\mathcal{F}})=\cohom^1([S(\mathcal{F}^c)],\mathcal{A}^1_{\mathcal{F}})=0$. Hence, by Theorem~1.1 in \cite{Parkgluing} the gluing problem has only the trivial solution.

Let $Q\le D$ be the largest ($\mathcal{F}$-centric) subgroup occurring in $\sigma$. Then as in the proof of Lemma~\ref{essential} we have $Q=(Q\cap\langle x,y\rangle)\times\langle z\rangle$. If $Q\cap\langle x,y\rangle$ is nonabelian, $\Aut(Q)$ is a $2$-group by Lemma~\ref{aut}. In this case we get $\cohom^i(\Aut_{\mathcal{F}}(\sigma),F^\times)=0$ for $i=1,2$ (see proof of Corollary~2.2 in \cite{Parkgluing}). Hence, we may assume that $Q\in\{Q_1,Q_2\}$ and $\Aut_{\mathcal{F}}(Q)\cong S_3$ (see proof of Lemma~\ref{subrep} for the case $Q=\langle x,z\rangle$). Then $\sigma$ only consists of $Q$ and $\Aut_{\mathcal{F}}(\sigma)=\Aut_{\mathcal{F}}(Q)$. Hence, also in this case we get $\cohom^i(\Aut_{\mathcal{F}}(\sigma),F^\times)=0$ for $i=1,2$.
\end{proof}

It seems likely that one can prove similar results about blocks with defect group $Q_{2^n}\times C_{2^m}$ or $SD_{2^n}\times C_{2^m}$, where $Q_{2^n}$ denotes the quaternion group and $SD_{2^n}$ denotes the semidihedral group of order $2^n$. This would generalize Olsson's results for $m=0$ (see \cite{Olsson}).

\section*{Acknowledgment}
I am very grateful to the referee for some valuable comments. 
This work was partly supported by the “Deutsche Forschungsgemeinschaft”.

\end{document}